\documentclass[journal,twoside,web]{ieeecolor}

\usepackage{generic}
\usepackage{cite}
\usepackage{amsmath,amssymb,amsfonts}
\usepackage{algorithmic}
\usepackage{graphicx}
\usepackage{textcomp}
\usepackage{mathbbol}
\DeclareMathOperator*{\argmin}{argmin}

\usepackage{comment}
\usepackage{balance}
\usepackage{graphicx}
\usepackage{subfigure}
\usepackage{epsfig} 
\usepackage{cite}
\usepackage{lcsys}

\usepackage[colorlinks = true,
            linkcolor = blue,
            urlcolor  = blue,
            citecolor = blue,
            anchorcolor = blue]{hyperref}

\usepackage{etoolbox}
\makeatletter
\@ifundefined{color@begingroup}%
  {\let\color@begingroup\relax
   \let\color@endgroup\relax}{}%
\def\fix@ieeecolor@hbox#1{%
  \hbox{\color@begingroup#1\color@endgroup}}
\patchcmd\@makecaption{\hbox}{\fix@ieeecolor@hbox}{}{\FAILED}
\patchcmd\@makecaption{\hbox}{\fix@ieeecolor@hbox}{}{\FAILED}

\newtheorem{theorem}{Theorem}
\newtheorem{lemma}{Lemma}
\newtheorem{proposition}{Proposition}

\newtheorem{definition}{Definition}
\newtheorem{remark}{Remark}

\usepackage[nameinlink]{cleveref}
\crefname{equation}{}{}
\crefname{theorem}{Theorem}{Theorems}
\crefname{corollary}{Corollary}{Corollaries}
\crefname{example}{Example}{Examples}
\crefname{assumption}{Assumption}{Assumptions}
\crefname{lemma}{Lemma}{Lemmas}
\crefname{proposition}{Proposition}{Propositions}
\crefname{figure}{Figure}{Figures}
\crefname{table}{Table}{Tables}
\crefname{fact}{Fact}{Facts}
\crefname{conjecture}{Conjecture}{Conjectures}
\crefname{section}{Section}{Sections}
\crefname{appendix}{Appendix}{Appendices}
\crefname{definition}{Definition}{Definitions}
\Crefname{equation}{}{}
\Crefname{theorem}{Theorem}{Theorems}
\Crefname{corollary}{Corollary}{Corollaries}
\Crefname{example}{Example}{Examples}
\Crefname{lemma}{Lemma}{Lemma}
\Crefname{proposition}{Proposition}{Proposition}
\Crefname{figure}{Figure}{Figures}
\Crefname{table}{Table}{Tables}
\Crefname{section}{Section}{Sections}
\Crefname{definition}{Definition}{Definitions}
\Crefname{appendix}{Appendix}{Appendices}

\newcommand{\ini}{\textnormal{ini}}
\newcommand{\col}{\textnormal{col}}
\newcommand{\f}{\textnormal{F}}
\newcommand{\m}{\textnormal{m}}
\newcommand{\ob}{\textnormal{o}}
\newcommand{\p}{\textnormal{P}}
\newcommand{\D}{\textnormal{d}}
\newcommand{\z}{\textnormal{z}}

\newcommand{\R}{\textnormal{r}}

\newcommand{\tr}{{{\mathsf T}}}

\newcommand{\method}[1]{\texttt{#1}}
\usepackage{siunitx}
\usepackage{threeparttable}
\usepackage{booktabs}
\renewcommand{\baselinestretch}{.985}

\begin{document}

\def\BibTeX{{\rm B\kern-.05em{\sc i\kern-.025em b}\kern-.08em
    T\kern-.1667em\lower.7ex\hbox{E}\kern-.125emX}}
\markboth{\journalname, VOL. XX, NO. XX, XXXX 2017}
{Author \MakeLowercase{\textit{et al.}}: Preparation of Papers for IEEE Control Systems Letters (August 2022)}

\title{Willems' Fundamental Lemma for Nonlinear Systems with Koopman Linear Embedding}

\author{
Xu Shang$^{1}$,
Jorge Cort\'es$^{2}$, \IEEEmembership{Fellow, IEEE}, 
and Yang Zheng$^{1}$, \IEEEmembership{Member, IEEE}
\thanks{This work is supported by NSF CMMI 2320697, NSF CAREER 2340713, and an Early Career Faculty Development Award from~the Jacobs School of Engineering, UC San Diego. The authors would like to thank Masih Haseli for discussions on Koopman operator~theory. }
\thanks{$^{1}$S. Xu and Y. Zheng are with the Department of Electrical and Computer Engineering, University of California San Diego; \texttt{\{x3shang,zhengy\}@ucsd.edu}}
\thanks{$^{2}$J. Cort\'es is with the Department of Mechanical and Aerospace Engineering, University of California San Diego, \texttt{cortes@ucsd.edu}}
\vspace{-5mm}
}


\maketitle
 \thispagestyle{plain}

\begin{abstract}
Koopman operator theory and Willems’ fundamental lemma both can provide (approximated) data-driven linear representation for nonlinear systems. However, choosing lifting functions for the Koopman operator is challenging, and the quality of the data-driven model from Willems’ fundamental lemma has no guarantee for general nonlinear systems. In this paper, we extend Willems’ fundamental lemma for a class of nonlinear systems that admit a \textit{Koopman linear embedding}. We first characterize the relationship between the trajectory space of a nonlinear system and that of its Koopman linear embedding. We then prove that the trajectory space of Koopman linear embedding can be formed by a linear combination of rich-enough trajectories from the nonlinear system. Combining these two results leads to a data-driven representation of the nonlinear system, which bypasses the need for the lifting functions and thus eliminates the associated bias errors. Our results illustrate that both the \textit{width} (more trajectories) and \textit{depth} (longer trajectories) of the trajectory library are important to ensure the accuracy of the data-driven model. 
\end{abstract}


\section{Introduction}
\label{sec:introduction}
Designing controllers for nonlinear systems with approximated linear representations has gained increasing interest. Linear approximations enable the utilization of linear system tools and facilitate computationally efficient model predictive control schemes. 
Both Koopman operator theory~\cite{koopman1931hamiltonian} and Willems' fundamental lemma \cite{willems2005note} can be applied~to~construct (approximated) linear representations of nonlinear systems from input and output data, which have shown promising performance in many practical applications \cite{haggerty2023control, elokda2021data,Wang2023-DeeP-LCC,shang2024decentralized}.

Koopman operator theory is originally developed for autonomous systems with no input \cite{koopman1931hamiltonian}. There are different~Koopman operator schemes to handle controlled nonlinear systems, \emph{e.g.}, taking the control sequence as an extended state~\cite{korda2018linear}~or considering the control sequence as extra parameters \cite{haseli2023modeling}. One key step is to \textit{lift} the state space into a higher-dimensional space, in which the lifted state evolves (approximately) in a linear way. This idea leads to a rigorous framework~of~Koopman operator theory for autonomous systems \cite{mezic2005spectral}, and extensions for controlled systems are under extensive development~\cite{mauroy2020koopman}. With (approximated) linear representations using the Koopman operator available, many control techniques have been applied, such as linear optimal control and model~predictive control \cite{korda2018linear,mauroy2020koopman}.  In all these methods, the accuracy of the Koopman-based approximations depends critically on the lifting functions, and choosing the right set is challenging~\cite{haseli2021learning}. 

Willems' fundamental lemma for~linear time-invariant (LTI) systems shows that a rich-enough trajectory library is sufficient to produce a direct data-driven representation \cite{willems2005note} of the system evolution. Its wide range of applicability has motivated the search for extensions, including special classes of nonlinear systems, such as Hamerstein and Wiener systems \cite{berberich2020trajectory}, bilinear systems \cite{yuan2022data,markovsky2022data}, and certain polynomial systems~\cite{guo2021data}. This data-driven representation can be utilized for linear controller design \cite{berberich2020robust} and also model predictive control \cite{coulson2019data}. When dealing with non-deterministic or nonlinear systems, it becomes necessary to include suitable regularization terms in predictive control to ensure its performance and increase the size of the trajectory library to construct a good data-driven representation~\cite{dorfler2022bridging, breschi2023data, shang2024convex}. Although the benefits of increasing the width of the trajectory library (\emph{i.e.}, collecting more trajectories) are well-recognized in the literature, the importance of enlarging its depth (\emph{i.e.}, extending the trajectory length) is less discussed.

In this paper, we aim to develop an extended Willems' fundamental lemma for nonlinear systems that admit a lifted linear representation under the Koopman operator, which we call \textit{Koopman linear embedding} (see \Cref{def:accu-Koop}). Unlike~previous studies \cite{korda2018linear,berberich2020trajectory, yuan2022data,markovsky2022data}, our direct data representation requires no prior knowledge of the lifting functions in the Koopman linear embedding or the nonlinearity of the system. One key idea in our approach is to establish an exact relationship between the trajectory spaces of the nonlinear system and its associated Koopman linear embedding. In general, the latter contains the former. We provide a necessary and sufficient condition for a trajectory of the Koopman linear embedding to be a trajectory of the original nonlinear system (\cref{prop:traj-rela}). Motivated by \cite[Def. 1]{Berberich2022}, we introduce a new notion of lifted excitation for nonlinear systems which accounts for the lifted state in Koopman linear embedding. We show the behavior of Koopman linear embedding can be fully captured by a linear combination of rich enough trajectories from the nonlinear system (\cref{prop:exist-traj-lib}). We finally establish a data-driven representation adapted from Willems' fundamental lemma for nonlinear systems with a Koopman linear embedding (\cref{theorem:Koopman-to-data-representation}). Thus, we can directly utilize the simple-to-build data-driven representation and bypass the need to choose lifting functions. 

Our data-driven representation can be directly utilized in predictive control. Our approach also illustrates the importance of the \textit{width} and \textit{depth} of the trajectory library, which depends on the ``hidden'' dimension of the Koopman linear embedding. Both collecting more trajectories (increasing the width) and utilizing longer initial trajectories (increasing the depth) are critical for the data-driven representation of nonlinear systems. 

The remainder of this paper is structured as follows.~\Cref{sec:preliminaries} reviews Koopman linear embedding and Willems' fundamental lemma. \Cref{sec:equal-DD-Koop} shows that a Koopman linear embedding of the nonlinear system leads to a direct data-driven representation. \Cref{sec:results} validates our theoretical findings via numerical simulations. We conclude the paper with \Cref{sec:conclu}. 

\textit{Notation:} 
Given a series of vectors $a_1, \ldots, a_n$ and matrices $A_1, \ldots, A_n$ with the same column dimension, we denote $\col(a_1, \ldots, a_n) := \begin{bmatrix}
    a_1^\tr, \ldots, a_n^\tr
\end{bmatrix}^\tr$ and $\col(A_1,\ldots, A_n) := \begin{bmatrix}
    A_1^\tr, \ldots, A_n^\tr
\end{bmatrix}^\tr$. We denote the quadratic form $a^\tr X a$ as $\|a\|_X^2$ 
and $\mathrm{diag}(b_1,...,b_n)$ as a diagonal matrix with $b_1,...,b_n$ at its diagonal entries. We use $\|A\|_F$ to represent the Frobenius norm of the matrix $A$. Collecting a length-$T$ data sequence $v = \col(v_0, \ldots, v_{T-1})$, we represent $v_{p:q} := \col(v_p, \ldots, v_q)$ where $p,q \in \mathbb{Z}$ and $T > q \ge p \ge 0$. 

\section{Preliminaries and Problem Statement}
\label{sec:preliminaries}
\subsection{Koopman linear models for nonlinear systems}
Consider a discrete-time nonlinear system
\begin{equation}
\label{eqn:nonlinear-1}
x_{k+1} = f(x_k, u_k), \quad
y_k = g(x_k, u_k),
\end{equation}
where $x_k \in \mathbb{R}^n, u_k \in \mathbb{R}^m$ and $y_k\in \mathbb{R}^p$ are the state,~input, and output of the system at time $k$, respectively. One key idea of Koopman operator is to lift the state $x_k$ of the original nonlinear system to a higher-dimensional space via a set~of~lifting functions (often referred to as observables) \cite{mezic2005spectral}, where the evolution of these observables becomes (approximately) linear. 

In this paper, we consider an important case of Koopman linear embedding for nonlinear systems.

\begin{definition}[Koopman Linear Embedding] \label{def:accu-Koop}
The nonlinear system \cref{eqn:nonlinear-1} admits a Koopman linear embedding if there exists a set of linearly independent lifting functions $\phi_1(\cdot),$ $ \ldots, \phi_{n_\z}(\cdot):\mathbb{R}^n \rightarrow \mathbb{R}$ such that the lifted state  
\begin{equation} \label{eq:Koopman-lifting}
\Phi(x_k):= \col(\phi_1(x_k), \ldots, \phi_{n_\z}(x_k)) \in \mathbb{R}^{n_\z},  
\end{equation}
propagates linearly along all trajectories of \eqref{eqn:nonlinear-1} and the output $y_k$ is a linear map of $\Phi(x_k)$ and $u_k$. 
\end{definition} 

For a nonlinear system admitting a Koopman linear embedding \cref{eq:Koopman-lifting}, the new lifted state $z_k = \Phi(x_k) \in \mathbb{R}^{n_\z}$ satisfies 
\begin{equation}
    \label{eqn:Koop-rep}
        z_{k+1} = A z_k + B u_k, \quad 
        y_k = C z_k + D u_k,
    \end{equation} 
with matrices $A, B, C$ and $D $ having appropriated dimensions. Note that we normally have $n_\z \gg n$ and the matrix pair $(A, B)$ and $(C, A)$ in \eqref{eqn:Koop-rep} may not be controllable or observable. Even when an exact Koopman linear embedding~does not exist, many existing studies (especially in predictive~control) often use the linear model \cref{eqn:Koop-rep} to approximate the dynamics of the observables \cref{eq:Koopman-lifting}; see \cite{korda2018linear} for details.  

After choosing the observables \cref{eq:Koopman-lifting}, we can compute the matrices $A, B, C$ and $D$ for the linear model \cref{eqn:Koop-rep} using extended dynamic model decomposition (EDMD) \cite{mauroy2020koopman}. 
We organize the measured input-state-output data sequence of \eqref{eqn:nonlinear-1} as 
\[
    \begin{aligned}
    \mathrm{X}  = \begin{bmatrix}
        x_0, \ldots, x_{n_\D-2} 
    \end{bmatrix}, \quad &\mathrm{X}^+ = \begin{bmatrix}
        x_1, \ldots, x_{n_\D-1}
    \end{bmatrix}, \\
    \mathrm{U} = \begin{bmatrix}
        u_0, \ldots, u_{n_\D-2}
    \end{bmatrix}, \quad & Y = \begin{bmatrix}
        y_0, \ldots, y_{n_\D-2}
    \end{bmatrix}.
    \end{aligned}
    \]
With the lifting functions \cref{eq:Koopman-lifting}, we compute the lifted state~as 
    \[
    \begin{aligned}
    \mathrm{Z}  = \begin{bmatrix}\Phi(x_0), \ldots, \Phi(x_{n_\D-2})
    \end{bmatrix}, \quad 
    \mathrm{Z}^+  = \begin{bmatrix}\Phi(x_1), \ldots, \Phi(x_{n_\D-1})\end{bmatrix}.
    \end{aligned}
    \]
Then, we obtain the matrices $A$, $B$, $C$ and $D$ via two least-squares approximations:
 \begin{equation} \label{eq:least-squares-model}
{\small  \begin{aligned}
   (A, B) & \in \argmin_{A, B} \; \|\mathrm{Z}^+ - A\mathrm{Z} - B \mathrm{U} \|_F^2, \\ 
    (C, D) & \in \argmin_{C, D} \; \|\mathrm{Y} - C \mathrm{Z} - DU \|_F^2. 
\end{aligned}
}
    \end{equation}
It is not necessary to collect the data points in sequence and we can also use data pairs $(x_i, x_i^+, u_i, y_i)$ where $x_i^+ = f(x_i, u_i), y_i = g(x_i, u_i), i = 0, \ldots, n_\D-1$ (see \cite{korda2018linear} for details).

The choice of observables affects \cref{eq:least-squares-model} significantly. Even if a Koopman linear embedding exists for \cref{eqn:nonlinear-1}, we may not know the correct observables \cref{eq:Koopman-lifting} for such a Koopman linear embedding. An inexact choice can lead to significant modeling errors \cite{haseli2021learning}. In the literature, common choices for~\cref{eq:Koopman-lifting} include Gaussian kernel, polyharmonic splines, and thin plate splines~\cite{mauroy2020koopman}. However, none of them can guarantee an exact linear model even when a Koopman linear embedding exists.   

\subsection{Willems' Fundamental Lemma}
\label{subsec:fundamental lemma}
Willems' fundamental lemma is established for linear time-invariant (LTI) system of the form 
\begin{equation}
    \label{eqn:LTI-system}
        x_{k+1}  = A_1 x_k + B_1 u_k,  \quad
        y_k = C_1x_k + D_1 u_k,
\end{equation}
where the state, input and output at time $k$ are denoted as $x_k \in \mathbb{R}^{\tilde{n}}$, $u_k \in \mathbb{R}^{\tilde{m}}$ and $y_k \in \mathbb{R}^{\tilde{p}}$, respectively. We consider system \cref{eqn:LTI-system} from the behavioral (\emph{i.e.}, trajectory) perspective. The key idea is that a linear combination of rich enough offline trajectories of \cref{eqn:LTI-system} can represent its whole trajectory~space. 

Let us recall the notion of persistent excitation \cite{willems2005note}. 
\begin{definition}[Persistently exciting]
\label{def:PE-condiiton}
The length-$T$ sequence $\omega = \col(\omega_0,\ldots, \omega_{T-1})$ is persistently exciting (PE) of order $L$ if its Hankel matrix
\begingroup
    \setlength\arraycolsep{2pt}
\def\arraystretch{0.85} 
{
\[ 
\mathcal{H}_L(\omega) = \begin{bmatrix}
    \omega_0 & \omega_1 & \cdots &  \omega_{T-L} \\
    \omega_1 & \omega_2 & \cdots & \omega_{T-L+1} \\
    \vdots & \vdots & \ddots & \vdots \\
    \omega_{L-1} & \omega_L &\cdots  & \omega_{T-1}
\end{bmatrix}
\]
 has full row rank.
}
\endgroup 

\end{definition}

\vspace{1mm}

With the pre-collected input-state-output data in sequence, \emph{i.e.}, $u_\D = \col(u_0, \ldots, u_{n_\D-1}), x_\D = \col(x_0, \ldots, x_{n_\D-1})$ and $y_\D = \col(y_0,\ldots, y_{n_\D-1})$,  the following Willems' fundamental lemma is adapted from \cite[Theorem 1]{van2020willems}.
\begin{lemma}[Willems' fundamental lemma]
\label{thm:fundamental-lemma}
    \!Consider~the LTI system \eqref{eqn:LTI-system}. Assume the Hankel matrix formed by its pre-collected trajectory $H_0 := \col(\mathcal{H}_1(x_{\D, 0:n_\D-L}), \mathcal{H}_L(u_{\D}))$ has full row rank. Then, a length-$L$ input-output data sequence $\col(u, y) \in \mathbb{R}^{(\tilde{m}+\tilde{p})L}$ is a valid trajectory of~\cref{eqn:LTI-system} if and only if there exists $g \in \mathbb{R}^{n_\D-L+1}$ such~that
    \[
    \begin{bmatrix}
        \mathcal{H}_L(u_\D) \\
        \mathcal{H}_L(y_\D)
    \end{bmatrix} g = 
    \begin{bmatrix}
        u \\ y
    \end{bmatrix}.
    \]
\end{lemma}
\vspace{5pt}

\Cref{thm:fundamental-lemma} does not require the controllability of \cref{eqn:LTI-system} since it directly imposes a full-rank condition on $H_0$ that involves the state sequence. If \eqref{eqn:LTI-system} is controllable, then persistent excitation of order $L + \tilde{n}$ for the input sequence $u_\D$ is sufficient to guarantee the full rank of $H_0$ \cite{willems2005note}. Utilizing \Cref{thm:fundamental-lemma}, we can build a data-driven representation for system~\eqref{eqn:LTI-system}. We use $u_\ini = \col(u_{k-T_\ini}, \ldots, u_{k-1})$ and $u_\f=\col(u_k, \ldots, u_{k+N-1})$ to represent the most recent past input trajectory of length-$T_\ini$ and the future input trajectory of length-$N$ 
and $L = T_{\ini}+N$ (similarly for $y_\ini, y_\f$). Let us partition the Hankel matrix by its first $T_\ini$ rows (\emph{i.e.}, $U_\p, Y_\p$) and the last $N$ rows (\emph{i.e.}, $U_\f, Y_\f$)~as 
\[ {\small
\begin{bmatrix}
    U_\p \\ U_\f
\end{bmatrix}
:= \mathcal{H}_L(u_\D), \quad 
\begin{bmatrix}
Y_\p \\ Y_\f
\end{bmatrix}
:= \mathcal{H}_L(y_\D)}.
\]

From \Cref{thm:fundamental-lemma}, $\col(u_\ini, y_\ini, u_\f, y_\f)$ is a valid trajectory of~\cref{eqn:LTI-system} if and only if there exists $g \in \mathbb{R}^{n_\D-T_\ini -N+1}$ such that 
\begin{equation}
\label{eqn:DD-rep}
\col(U_\p, Y_\p, U_\f, Y_\f) g = \col(u_\ini, y_\ini, u_\f, y_\f). 
\end{equation}
Furthermore, if \cref{eqn:LTI-system} is observable and $T_\ini$ is no smaller than its observability index, then $y_\f$ in \cref{eqn:DD-rep} is unique given an initial trajectory $(u_{\ini}, y_{\ini})$ and any future input $u_\f$ \cite{coulson2019data}. Intuitively, if \cref{eqn:LTI-system} is observable, the initial trajectory $(u_{\ini}, y_{\ini})$ allows us to uniquely determine the corresponding initial state. This data-driven representation \cref{eqn:DD-rep} has been widely used in predictive control \cite{coulson2019data} with many successful applications \cite{shang2024decentralized,Wang2023-DeeP-LCC,elokda2021data}.   

\subsection{Problem Statement}
In this paper, we aim to extend the data-driven representation \cref{eqn:DD-rep} from LTI systems to nonlinear systems \cref{eqn:nonlinear-1} that admit a Koopman linear embedding. One may be tempted to directly apply Willems' fundamental lemma to the Koopman linear model \cref{eqn:Koop-rep} and get a similar data-driven representation as \cref{eqn:DD-rep}. However, there are two unsolved challenges for this process: 
   1) the Koopman linear model \cref{eqn:Koop-rep} may be neither controllable nor observable; 
  2)  the behavior space of the Koopman linear model~\cref{eqn:DD-rep} is much larger than the behavior space of the original nonlinear system~\cref{eqn:nonlinear-1}. 

We propose two innovations to resolve the challenges above. 1) We first characterize the relationship between the behavior space of the Koopman linear model \cref{eqn:DD-rep} and that of the original nonlinear system \cref{eqn:nonlinear-1}. A key insight of this characterization is to guarantee the uniqueness of the subsequent output given the leading input-output trajectory and the subsequent input, and observability is not needed for ensuring the uniqueness as long as the length of the initial trajectory is large enough, \emph{i.e.}, the Hankel matrix has a sufficient \textit{depth}. 2) Motivated by lack of controllability, we introduce a new notion of lifted excitation for the offline data collection, which has a similar flavor to \cite[Def. 1]{Berberich2022} that focuses on a special case of affine systems. With these two technical tools, we establish a direct data-driven representation in the form of \cref{eqn:DD-rep} for nonlinear systems that admit a Koopman linear embedding. This representation requires no knowledge of the lifting functions \cref{eq:Koopman-lifting} (as long as they exist). Our representation can be directly utilized in Koopman model predictive control \cite{korda2018linear}, without the need of identifying the linear model \cref{eqn:Koop-rep}. This bypasses the challenging problem of selecting the lifting functions, with the added remarkable benefit of eliminating the associated bias errors.

\section{From Koopman Linear Embeddings to Data-driven Representations}
\label{sec:equal-DD-Koop}
In this section, we develop the main technical result that directly represents the nonlinear system with Koopman linear embedding using its input and output data. 
We also discuss a special case of affine systems considered in~\cite{Berberich2022}.  

\subsection{Two behavior spaces}
\label{subsec:beh-space}
Consider the space of length-$L$  trajectories for the nonlinear system \cref{eqn:nonlinear-1} and the Koopman linear embedding~\cref{eqn:Koop-rep}:
\begin{subequations} \label{eq:T-trajectories-space}
  \begin{align}
        \mathcal{B}_1\!\!\mid_L &\! =\!\! \left\{\!\begin{bmatrix}
            u\\y
        \end{bmatrix}\!\in\! \mathbb{R}^{(m+p)L}  \mid  \exists\, x(0)\!=\!x_0 \! \in \! \mathbb{R}^n, \cref{eqn:nonlinear-1} \; \text{holds} \!\right\}\!, \label{eqn:traj-nonlinear}\\
         \mathcal{B}_2\!\!\mid_L &\! =\!\! \left\{\!\begin{bmatrix}
            u\\y
        \end{bmatrix}\! \in \! \mathbb{R}^{(m+p)L}  \mid  \exists\, z(0)\!=\!z_0 \! \in \! \mathbb{R}^{n_\z}, \cref{eqn:Koop-rep} \; \text{holds} \!\right\}\!. \label{eqn:traj-linear} 
    \end{align}
\end{subequations}
Note that $\mathcal{B}_1\!\!\mid_L $ is a nonlinear set while $\mathcal{B}_2\!\!\mid_L $ is a linear subspace in $\mathbb{R}^{(m+p)L}$. Intuitively, the behavior space of the Koopman linear embedding is larger than that of the original nonlinear system. Our first result characterizes the relationship between these two behavior spaces. 

\begin{theorem}
\label{prop:traj-rela}
Consider the nonlinear system \eqref{eqn:nonlinear-1} and assume it admits a Koopman linear embedding~\eqref{eqn:Koop-rep}.
\begin{enumerate}
    \item We have $\mathcal{B}_1\!\!\mid_L \; \subset \; \mathcal{B}_2\!\!\mid_L, \forall L \geq 1$, \emph{i.e.}, all trajectories of system \eqref{eqn:nonlinear-1} are also trajectories of~\eqref{eqn:Koop-rep};
    \item Let $\col(u,y) \in \mathcal{B}_2\!\!\mid_L$, where $L > n_\z$. Then, $\col(u,y) \in \mathcal{B}_1\!\!\mid_L$ if and only if its leading sequence of length $n_\z$ (\emph{i.e.}, $\col(u_{0:n_\z-1}, y_{0:n_\z-1})$) is a valid trajectory of~\eqref{eqn:nonlinear-1}.
\end{enumerate}
\end{theorem}

\Cref{prop:traj-rela} reveals that while the space $\mathcal{B}_2\!\!\mid_L $ is larger, we can characterize its subset corresponding to $\mathcal{B}_1\!\!\mid_L$ using the initial leading sub-sequence $\col(u_{0:n_\z-1}, y_{0:n_\z-1})$. We need a technical lemma to prove the second statement in \Cref{prop:traj-rela}. 

\begin{lemma}
\label{prop:traj-linear}
Consider an LTI system \cref{eqn:LTI-system}. Fix an initial trajectory $\col(u_{0:L_1-1}, y_{0:L_1-1}) \in \mathbb{R}^{(\tilde{m}+\tilde{p})L_1}$ of length-$L_1$, where $L > L_1 \geq \tilde{n}$. Given any subsequent input $u_{L_1:L-1} \in \mathbb{R}^{\tilde{m}(L-L_1)}$ (future input), the subsequent output $y_{L_1:L-1} \in \mathbb{R}^{\tilde{p}(L-L_1)}$ (future output)~is~unique.
\end{lemma}

Note that \Cref{prop:traj-linear} works for any LTI systems and requires no observability or controllability. This result is not difficult to establish and its proof is provided in \Cref{append:lemma2}. 

\textbf{Proof of \Cref{prop:traj-rela}:} 
   The first statement is obvious from \Cref{def:accu-Koop}. Let $\col(u,y) \in \mathcal{B}_1\!\!\mid_L$ be arbitrary. By definition, we can find $x_0 \in \mathbb{R}^n$ such that $\col(u,y)$ satisfies the evolution in \cref{eqn:nonlinear-1}. Then, with the lifted initial state 
   $z_0 = \Phi(x_0) \in \mathbb{R}^{n_\z},$ 
   $\col(u,y)$ satisfies the evolution in \cref{eqn:Koop-rep}. Thus, $\col(u,y) \in \mathcal{B}_2\!\!\mid_L$. 
    
    For the second statement, the ``only if'' part is trivial as $\col(u_{0:n_\z-1}, y_{0:n_\z-1})$ is part of $\col(u, y)$. We here prove ``if'' part. 
    Suppose $\col(u_{0:n_\z-1}, y_{0:n_\z-1}) \in \mathcal{B}_1\!\!\mid_{n_\z}$ and we let $\tilde{y}_{n_\z:L-1}$ be the corresponding outputs from the nonlinear system \cref{eqn:nonlinear-1} for the rest of inputs $u_{n_\z:L-1}$, \emph{i.e.},
    $$
    \col(u, y_{0:n_\z-1}, \tilde{y}_{n_\z:L-1}) \in \mathcal{B}_1\!\!\mid_{L}.
    $$
    Then, it is clear that $\col(u, y_{0:n_\z-1}, \tilde{y}_{n_\z:L-1})\in \mathcal{B}_2\!\!\mid_{L}$ utilizing the first statement. From \Cref{prop:traj-linear}, the outputs of the linear system \cref{eqn:Koop-rep} are uniquely determined by $u_{n_\z:L-1}$ when $\col(u_{0:n_\z-1}, y_{0:n_\z-1}) \in \mathcal{B}_2\!\!\mid_{n_\z}$ are given. Thus, we must have 
    $$
    \tilde{y}_{n_\z:L-1} = y_{n_\z:L-1},
    $$
    indicating the whole trajectory satisfies $\col(u,y)\!\in\!\mathcal{B}_1\!\!\mid_{L}$.
\hfill {\footnotesize $\blacksquare$}

\vspace{2pt}

We might be tempted to estimate an initial state $x_0$~or~$z_0$ from $\col(u_{0:n_\z-1}, y_{0:n_\z-1}) \in \mathcal{B}_1\!\!\mid_{n_\z}$. However, since~we~do~not assume observability of the Koopman linear embedding \Cref{eqn:Koop-rep}, the initial state $z_0$ cannot be uniquely determined from $\col(u_{0:n_\z-1}, y_{0:n_\z-1})$. As confirmed in \Cref{prop:traj-linear}, the unobservable part of the initial state does not affect the uniqueness of the input-output trajectory. If the Koopman linear embedding is observable with observability index $l_\ob$, the required length of the leading sequence can be decreased to $l_\ob$. In this case, the initial state $z_0$ can be estimated, and the uniqueness of the input-output trajectory is also guaranteed. The initial trajectory can thus be shorter as $T_\ini \ge l_\ob$ in \Cref{theorem:Koopman-to-data-representation}.

\subsection{Data-driven representation of nonlinear systems}
\label{subsec:DD-representation}
While the trajectory space of the Koopman linear embedding \Cref{eqn:Koop-rep} is larger than that of the nonlinear system \Cref{eqn:nonlinear-1}, we~can use the trajectories from~\Cref{eqn:nonlinear-1} (\emph{i.e.}, $\col(u_\D^i,y_\D^i) \in  \mathcal{B}_1\!\!\mid_{L}, i = 1, \ldots, l$) that are rich enough to  represent $\mathcal{B}_2\!\!\mid_{L}$. For this, we propose the following definition of lifted excitation. 
\begin{definition} \label{def:persistent-excitation-nonlinear}
Consider a nonlinear system  \eqref{eqn:nonlinear-1} with a Koopman linear embedding \eqref{eqn:Koop-rep}. We say $l$ trajectories of length-$L$ from \eqref{eqn:nonlinear-1}, $\col(u_\D^i,y_\D^i) \in  \mathcal{B}_1\!\!\mid_{L}, i = 1, \ldots, l$ with $l \geq mL + n_\z$ provide lifted excitation of order $L$, if the following matrix 
    \begin{equation} \label{eq:exitation-assumption}
    H_K := \begin{bmatrix}
        u_\D^1 & u_\D^2 & \ldots & u_\D^l \\
        \Phi(x_0^1) & \Phi(x_0^2) & \ldots & \Phi(x_0^l)
    \end{bmatrix} \in \mathbb{R}^{(mL + n_\z) \times l}
    \end{equation}
    has full row rank, where $x_0^i \in \mathbb{R}^n$ is the initial state for each trajectory $\col(u_\D^i,y_\D^i), i = 1, \ldots, l$.
\end{definition}

This notion of lifted excitation generalizes \cite[Def. 1]{Berberich2022}, that focuses only on affine systems. Our notion is suitable for any nonlinear system with a Koopman linear embedding. If \cref{eqn:Koop-rep} is controllable, we can design the inputs to ensure the lifted excitation. Given a single trajectory $u_\D, y_\D$ with length-$n_\D$, we obtain $l$ trajectories with length-$L$ using column vectors of $\col(\mathcal{H}_L(u_\D), \mathcal{H}_L(y_\D))$ where $l = n_\D-L+1$. If  $u_\D$ is PE of order $L+n_\z$, the $l$ trajectories above satisfy lifted excitation. If we collect multiple trajectories $u_\m^i, y_\m^i, i=1,\ldots, q$ (see \cite{van2020willems}), we require $u_\m^1, \ldots, u_\m^q$ to be collectively PE of order $L+n_\z$, i.e., 
$
\mathrm{rank}\big(\begin{bmatrix}
        \mathcal{H}_{L+n_z}(u_\m^1), \ldots, \mathcal{H}_{L+n_z}(u_\m^q)
    \end{bmatrix}\big) = m(L+n_z).
$   
The column vectors of $\col(\mathcal{H}_L(u_\m^i), \mathcal{H}_L(y_\m^i)),$ $i= 1,\ldots, q$ satisfies the lifted excitation in \Cref{def:persistent-excitation-nonlinear}.

\begin{theorem}
\label{prop:exist-traj-lib}
    Consider a nonlinear system  \eqref{eqn:nonlinear-1} with Koopman linear embedding \eqref{eqn:Koop-rep}. Suppose that $l$ trajectories of length-$L$ from \eqref{eqn:nonlinear-1}, $\col(u_\D^i,y_\D^i) \in  \mathcal{B}_1\!\!\mid_{L}, i = 1, \ldots, l$, satisfy lifted excitation of order $L$.
    Then, a length-$L$ $\col(u, y)$ is a valid trajectory of the Koopman linear embedding \eqref{eqn:Koop-rep}  if and only if there exists $g \in \mathbb{R}^l$ such that 
    $H_\D  g = 
        \col(u, y)$ 
where 
         \begin{equation} \label{eq:trajectory-library-Hankel}
     H_\D := \begin{bmatrix}
        u_\D^1 & u_\D^2 & \ldots & u_\D^l \\
        y_\D^1 & y_\D^2 & \ldots & y_\D^l
    \end{bmatrix} \in \mathbb{R}^{(m+p)L \times l}. 
    \end{equation}
\end{theorem}

\vspace{5pt}

Due to the page limit, we postpone the proof to \Cref{append:them2}. \Cref{prop:exist-traj-lib} allows us to use the trajectories from the nonlinear system \cref{eqn:nonlinear-1} that provide lifted excitation of order $L$ (always exist; see \Cref{append:exist-traj-lib}) to represent any length-$L$ trajectory of the Koopman linear embedding \cref{eqn:Koop-rep}. Given trajectories of the nonlinear system, it is still an open problem to determine whether a Koopman linear embedding exists (see details in \Cref{append:dis-KLE}). Combining \Cref{prop:traj-rela,prop:exist-traj-lib} leads to a direct data-driven representation for a nonlinear system 
\cref{eqn:nonlinear-1} with Koopman linear embedding \cref{eqn:Koop-rep}, as we describe in our next result. Given a trajectory library $H_\D \!= \! \col(U_\D, Y_\D)$ in \Cref{eq:trajectory-library-Hankel}, where each column is a trajectory of length $L \!= \!T_\ini \! +\! N$ from the nonlinear system \cref{eqn:nonlinear-1}, we partition matrices $U_\D$ and~$Y_\D$ as 
\begin{equation}
\label{eqn:partition}
\begin{bmatrix}
    U_{\p} \\
    U_{\f} 
\end{bmatrix}  := U_\D, \quad
\begin{bmatrix}
    Y_{\p} \\
    Y_{\f} 
\end{bmatrix} := Y_\D,
\end{equation}
where $U_{\p}$ and $U_{\f}$ consist of the first $T_\textnormal{ini}$ rows and the last $N$ rows of $U_\D$, respectively (similarly for $Y_{\p}$ and $Y_{\f}$). 
\begin{theorem}
\label{theorem:Koopman-to-data-representation}
Consider a nonlinear system  \eqref{eqn:nonlinear-1} with a Koopman linear embedding \eqref{eqn:Koop-rep}. We collect a data library $H_\D$  in \cref{eq:trajectory-library-Hankel} with $l \geq mL +n_\z$ trajectories, whose length $L$ is $T_\ini + N$ and $T_\ini \geq n_\z$. Suppose these $l$ trajectories have lifted excitation of order $L$. At time $k$, denote the most recent input-output sequence $\col(u_{\textnormal{ini}}, y_\ini)$ with length-$T_{\textnormal{ini}}$ from \cref{eqn:nonlinear-1} as 
\[
 u_{\ini} = \col(u_{k-T_\ini}, \ldots,u_{k-1}), \quad
 y_{\ini} = \col(y_{k-T_\ini}, \ldots,y_{k-1}).
\]
For any future input $u_\f = \col(u_k, \ldots,u_{k+N-1})$, the sequence $\col(u_\ini, y_\ini, u_\f, y_\f)$ is a valid length-$L$ trajectory of \cref{eqn:nonlinear-1} if and only if there exists $g \in \mathbb{R}^{l}$ such that 
\begin{equation} \label{eq:data-representation-Koopman}
\col(U_\p, Y_\p, U_\f,  Y_\f) g
= 
\col(u_\ini, y_\ini, u_\f, y_\f). 
\end{equation}
\end{theorem}

\vspace{2pt}
This result is a combination of \Cref{prop:exist-traj-lib,prop:traj-rela}. Due to the page limit, we provide some details in \Cref{append:them3}. \Cref{theorem:Koopman-to-data-representation} gives a direct data-driven representation~of~nonlinear systems with a Koopman linear embedding from~its~input and output data. This data-driven representation requires no knowledge of the lifting functions \cref{eq:Koopman-lifting} as long as they exist. Also, we do not require the Koopman linear embedding \eqref{eqn:Koop-rep} to be controllable or observable. Two key enablers for \Cref{theorem:Koopman-to-data-representation} are 1) our notion of lifted excitation for nonlinear systems in \cref{def:persistent-excitation-nonlinear} that enables \Cref{prop:exist-traj-lib}, and 2) a sufficiently long initial trajectory $\col(u_\ini, y_\ini)$ from the nonlinear system that ensures \Cref{prop:traj-rela}. When $n_\z$ is large, a great amount of data is needed to satisfy conditions for both enablers (see the effect of insufficient data in \Cref{append:them3}). 

We here remark that \Cref{theorem:Koopman-to-data-representation} illustrates the importance of increasing the \textit{width} and \textit{depth} of the trajectory library \cref{eq:trajectory-library-Hankel,eqn:partition}. While the benefits of increasing width are well-recognized in the literature, the importance of enlarging the depth has been overlooked. Collecting more trajectories to increase the width of \cref{eq:trajectory-library-Hankel} contributes to the lifted excitation condition (see \Cref{prop:exist-traj-lib}). On the other hand, fixing the prediction horizon $N$, a sufficient \textit{depth} ensures the initial trajectory is long enough in \cref{eqn:partition}, which guarantees that the trajectory in the space of the Koopman linear embedding is also a valid trajectory for the nonlinear system (see \Cref{prop:traj-rela}). Furthermore, \Cref{theorem:Koopman-to-data-representation} shows the required width and depth of the trajectory library depend on the ``hidden" dimension of the Koopman linear embedding of the nonlinear system. We note that estimating the Koopman order $n_z$ is non-trivial (see detailed discussion in \Cref{append:dis-KLE} and \Cref{append:nonexist-SC}). 

According to \Cref{theorem:Koopman-to-data-representation}, the data-driven representation \cref{eq:data-representation-Koopman} is equivalent to the Koopman linear embedding. This can be directly integrated with predictive control at each time $k$ as
\begin{equation}
\label{eqn:pred-controller}
\begin{aligned}
\min_{u_\f \in \mathcal{U}, y_\f} \quad &  \|u_\f\|_R^2 + \|y_\f-y_\R\|_Q^2 \\
\mathrm{subject~to} \quad & \eqref{eq:data-representation-Koopman}
\end{aligned}
\end{equation}
where $R \succ 0, Q \succeq 0$, $y_\R$ denotes the reference output trajectory, and $u_\f \in \mathcal{U}$ is the input constraint. For nonlinear systems with Koopman linear embedding, there is no need to use any lifting functions, which are instead required by most existing Koopman-based model predictive control approaches~\cite{korda2018linear,mauroy2020koopman}. 

\begin{remark}
    The literature on extending Willems’ fundamental lemma to nonlinear cases \cite{berberich2020trajectory,yuan2022data,markovsky2022data} requires system~dynamics, with additional constraints for their nonlinear structures. 
    When the nonlinearities exist in input or output (\emph{e.g.}, Hamerstein systems, Wiener systems), a change of variables is needed \cite{berberich2020trajectory}. Our data-driven representation \cref{eq:data-representation-Koopman} has no additional constraints and requires no knowledge of lifting functions. We only need a lifted excitation condition and a sufficiently long initial trajectory. The work \cite{lian2021koopman} integrates Willems' fundamental lemma with the learning of lifting~functions and \cite{lazar2024basis} forms the trajectory library with lifted states, which also requires learning lifting functions. In contrast, we show that learning these lifting functions is redundant for nonlinear systems with Koopman linear~embedding. 
\end{remark}

\subsection{A special case: Affine systems}
\label{subsec:}
Consider an affine system of the form 
\begin{equation}
    \label{eqn:affine-sys}
    \begin{aligned}
    x_{k+1} = Ax_k + Bu_k + e, \quad
    y_k  = Cx_k + D u_k + r,
    \end{aligned}
\end{equation}
where $x_k \in \mathbb{R}^{\bar{n}}$, $u_k \in \mathbb{R}^{\bar{m}}, y_k \in \mathbb{R}^{\bar{p}}$ are the state, input and output at time $k$, respectively and $e \in \mathbb{R}^{\bar{n}}, r \in \mathbb{R}^{\bar{p}}$ are two constant vectors. The result in \cite[Theorem 1]{Berberich2022} presents a data representation for~\cref{eqn:affine-sys}, which we reproduce below. 

\begin{theorem}[{\!\cite[Theorem 1]{Berberich2022}}]
\label{thm:fundamental-lemma-affine}
    Given the pre-collected trajectories $u_\D,x_\D,y_\D$ of \eqref{eqn:affine-sys}, suppose
    \begin{equation}\label{eq:persistent-excitation-affine}
        H_A \!:= \!\col(\mathcal{H}_L(u_\D),\! \mathcal{H}_1(x_{\D,0:n_\D-L}),\! \mathbb{1})
    \end{equation}
    has full row rank. Then, a length-$L$ input-output data sequence $\col(u_\ini, y_\ini, u_\f, y_\f)\in \mathbb{R}^{(\bar{m}+\bar{p})L}$ is a valid trajectory of \eqref{eqn:affine-sys} if and only if there exists $g \in \mathbb{R}^{n_\D-L+1}$ such that
    \begin{equation} 
\label{eq:data-representation-affine}
\col(
    U_\p, Y_\p, U_\f,  Y_\f) g= 
\col(u_\ini, y_\ini, u_\f, y_\f), \ \sum g = 1.
\end{equation}
\end{theorem}

\vspace{3pt} 
A unique feature in \Cref{thm:fundamental-lemma-affine} is that the coefficient $g$ should have an affine constraint, since \cref{eqn:affine-sys} is affine. Here, we show that any affine system \cref{eqn:affine-sys} has an exact Koopman linear embedding, and thus our main \Cref{theorem:Koopman-to-data-representation} naturally applies. Choose a vector of lifting functions $z(x) = [\phi_1(x), \ldots, \phi_n(x), 1]^\tr$ where $\phi_i(x): \mathbb{R}^{\bar{n}} \rightarrow \mathbb{R}$ is the $i\textrm{-th}$ element of the state, \emph{i.e.}, $\phi_i (x)= x_i$. Then, we have a Koopman linear embedding for the affine system~\eqref{eqn:affine-sys} as  
\begin{equation}
\label{eqn:linear-pred}
    \begin{aligned}
        z_{k+1} \! = \! \begin{bmatrix}
            A & e \\ \mathbb{0} & 1
        \end{bmatrix} z_k+\begin{bmatrix}
            B \\ \mathbb{0}
        \end{bmatrix} u_k, \;\; 
        y_k \!=\! \begin{bmatrix}
            C & r 
        \end{bmatrix} z_k + D u_k. \\
    \end{aligned}
\end{equation}

Consequently, with the lifted excitation and an initial trajectory of length $\bar{n}+1$ in \Cref{theorem:Koopman-to-data-representation}, the data-representation \cref{eq:data-representation-Koopman} is necessary and sufficient for the behavior of \eqref{eqn:affine-sys}. In this case, we need a slightly longer initial trajectory, but no affine constraint on $g$ is needed. We note that the data matrices $H_A$ in \Cref{eq:persistent-excitation-affine} and $H_K$ in \Cref{eq:exitation-assumption} become the same, thus 
the lifted excitation and the data requirement in \cref{theorem:Koopman-to-data-representation,thm:fundamental-lemma-affine} are identical. Our data-representation in \Cref{theorem:Koopman-to-data-representation} works for any nonlinear systems with Koopman linear embedding while it is unclear how to design constraints on $g$ to extend \Cref{thm:fundamental-lemma-affine}. 
\begin{remark}
    General nonlinear systems may not admit a Koopman linear embedding, even in infinite dimensions \cite{iacob2024koopman}. However, Koopman linear models are widely used to approximate nonlinear dynamics and have demonstrated good performance in real applications when the approximation error is small \cite{haggerty2023control, Wang2023-DeeP-LCC,elokda2021data}. 
    In \cite{Berberich2022}, the affine data-driven representation \cref{eq:data-representation-affine} is continuously updated using the most recent trajectory, approximating the local linearization of the nonlinear system. It is an interesting direction to integrate \cref{theorem:Koopman-to-data-representation} with the approach in \cite{Berberich2022} for a broader class of nonlinear systems. 
\end{remark}

\vspace{-2mm}
\section{Numerical Experiments}
\label{sec:results}
In this section, we present numerical experiments to compare the prediction and control performance of four linear representations: 1) our proposed \underline{D}ata-\underline{D}riven \underline{K}oopman representation (\method{DD-K}) \eqref{eq:data-representation-Koopman}, 2) \underline{D}ata-\underline{D}riven \underline{A}ffine representation (\method{DD-A}) \eqref{eq:data-representation-affine}, 3) the standard Koopman linear approximation \eqref{eqn:Koop-rep} from EDMD~\eqref{eq:least-squares-model} (\method{EDMD-K}) and 4) the \underline{D}eep \underline{N}eural \underline{N}etwork \underline{K}oopman representation (\method{DNN-K}) in \cite{shi2022deep}.

\vspace{-2mm}
\subsection{Experiment Setup}
We consider the following nonlinear system 
\[
\begin{aligned}
\begin{bmatrix}
    x_{1,k+1} \\
    x_{2,k+1}
\end{bmatrix} & = 
\begin{bmatrix}
    0.99 x_{1,k} \\
    0.9 x_{2,k} + x_{1,k}^2 + x_{1,k}^3 + x_{1,k}^4 + u_k
\end{bmatrix},
\end{aligned}
\]
with output $y_k =  x_k$, and state $x = \col(x_1, x_2) \in \mathbb{R}^2$ and input $u \in \mathcal{U} := \begin{bmatrix}
    -5, 5
\end{bmatrix}$.
We choose the lifted state as $z := \col(x_1, x_2, x_1^2, x_1^3, x_1^4)$, and its Koopman linear embedding~is
\[
\begingroup
    \setlength\arraycolsep{2pt}
\def\arraystretch{0.85} 
\begin{aligned}
z_{k+1} & = \begin{bmatrix}
    0.99 & 0 & 0 & 0 & 0  \\
    0 & 0.9 & 1 & 1 & 1  \\
    0 & 0 & 0.99^2 & 0 & 0  \\
    0 & 0 & 0 & 0.99^3 & 0  \\
    0 & 0 & 0 & 0 & 0.99^4 
\end{bmatrix}z_k + \begin{bmatrix}
    0 \\ 1 \\ 0\\ 0\\ 0
\end{bmatrix}u_k, 
\end{aligned}
\endgroup
\]
and $y_k = z_k(1:2)$. This linear embedding is not controllable but has an observability index $4$. In our experiments, the prediction horizon is $N = 20$, and the lengths of the initial trajectory $\col(u_\ini, y_\ini)$ are $4$ and $2$ for \method{DD-K} and \method{DD-A} (the entire trajectory length $L$ is 24 for \method{DD-K} and 22 for \method{DD-A}), respectively. Then, we collect a single trajectory of length-$52$ and obtain $29$ length-$24$ trajectories for \method{DD-K}, which is the minimum necessary data length to make $H_K$ in \cref{eq:exitation-assumption} a square matrix. For the EDMD method, we simulate $200$ trajectories with $200$ time steps. The lifting functions are chosen to be the state of~\eqref{eqn:nonlinear-1} and $300$ thin plate spline radial basis functions whose center $x_0$ is randomly selected with uniform distribution from $\begin{bmatrix} -1, 1 \end{bmatrix}^2$ and has the form $\phi(x) = \|x-x_0\|_2^2\mathrm{log}(\|x-x_0\|_2)$. To learn the lifting functions with DNN, we collect $10^5$ trajectories with $100$ steps, and we used $2$ hidden layers with $64$ units to learn $32$ lifting functions (\emph{i.e.}, the DNN's output is $32$). The parameters in \cref{eqn:pred-controller} are set as $R = I_N$ and $Q = I_N \otimes \mathrm{diag}(0, 100)$, and the prediction model is replaced by \cref{eqn:Koop-rep} and \cref{eq:data-representation-affine} for \method{EDMD-K} and \method{DD-A}, respectively. A regularization $\mathrm{reg} = \lambda_g \|g\|_2 + \lambda_y \|\sigma_y\|_2$ is added for the \method{DD-A} with variables $g, \sigma_y$ and $\lambda_g  =400, \lambda_y = 2 \times 10^5$ to ensure feasibility and numerical stability.

\vspace{-2mm}
\subsection{Prediction and Control Performance}
\label{subsec:accurate-pred}
We first compare the prediction performance for the four methods, and we also illustrate our \method{DD-K} with different initial trajectory lengths. Given the future input $u_\f(k) = 5 \sin(\pi k/4)$, \Cref{fig:comp_Meth} displays results for the four methods. The predicted output trajectories of \method{DD-K} with different initial trajectory lengths are shown in \Cref{fig:comp_Ini}. As expected from \Cref{theorem:Koopman-to-data-representation}, the predicted trajectory from \method{DD-K} with $T_{\ini} = 4$ is the same as the true trajectory (see red and black dashed curves in \Cref{fig:pred_performance}). However, trajectories from \method{DD-A} and \method{EDMD-K} (see orange and blue curves in \Cref{fig:comp_Meth}) and the \method{DD-K} with initial trajectory length $2$ and $3$ (see purple and brown curves in \Cref{fig:comp_Ini}) deviate from the true trajectory. The trajectory from \method{DNN-K} is relatively close to the true trajectory (see the green curve in \Cref{fig:comp_Meth}) but its performance varies significantly for different pre-collected data sets. For \method{DD-A}, the affine constraint is inaccurate for this non-affine system. For \method{EDMD-K} and \method{DNN-K}, their lifting functions are not guaranteed to form an invariant space (cf.~\cite{haseli2021learning}), and both EDMD and DNN have approximation~errors.

\begin{figure}[t]
\centering
\setlength{\abovecaptionskip}{2pt}
\subfigure[Comparison of linear models]{\includegraphics[width=0.22\textwidth]{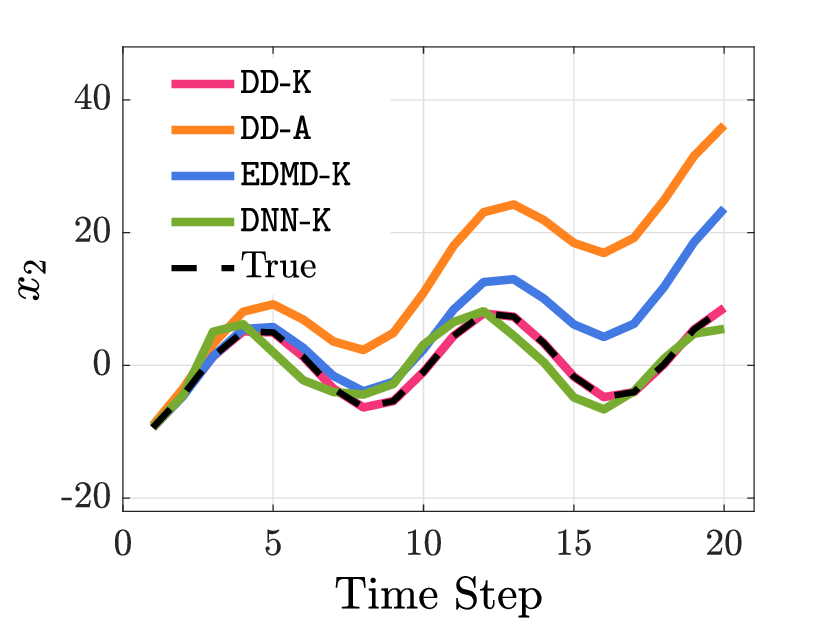} \label{fig:comp_Meth}}
\hspace{-1mm} \subfigure[Comparison of different $T_{\ini}$]{\includegraphics[width=0.22\textwidth]{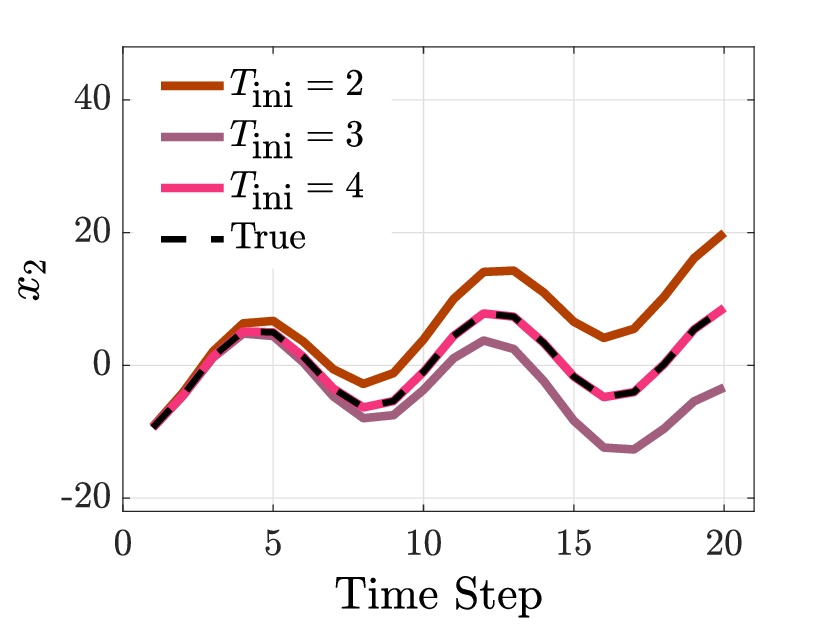} \label{fig:comp_Ini}} 
\caption{Prediction of $x_2$ with a given input $u_\f$. In (a) and (b), the red curve is the predicted trajectory of \method{DD-K} with $T_{\ini} = 4$. 
The orange, blue and green curves in (a) are \method{DD-A}, \method{EDMD-K} and \method{DNN-K}, and the brown and purple curves in (b) are \method{DD-K} with initial trajectory of $T_{\ini} = 2$ and~$3$.}
\label{fig:pred_performance}
\vspace*{-2ex}
\end{figure}

We next compare control performance of predictive controllers utilizing different linear representations. Our goal is to make $x_2$ track two types of reference trajectories: 1) Sinusoidal wave $y_{\R,k} =\col(0, 5\sin(\pi k / 30))$ and 2) Step signal $y_\R = \col(0, 5)$. We consider the realized control cost that is computed as $\|u^*\|_R^2 + \|y^*-y_r\|_Q^2$ where $u^*$ is the computed control input and $y^*$ is the actual trajectory after applying $u^*$. The results are shown in \Cref{fig:control}, and the realized control cost is averaged over $100$ data sets since the performance of these models is related to the pre-collected data. From \Cref{fig:comp_control}, we can observe the controller with \method{DD-K} can track the reference trajectory perfectly (see red and black dashed curves). \method{EDMD-K} and \method{DNN-K} can also track the reference trajectory closely and \method{DNN-K} converges faster to the reference trajectory (see green and blue curves) while applying \method{DD-A} has a longer transition phase. The realized control cost in \Cref{fig:comp_cost} further demonstrates \method{DD-A} $>$ \method{EDMD-K} $>$ \method{DNN-K} $>$ \method{DD-K} for both sinusoidal and step signals. We remark~that models obtained from DNN have large variances and we eliminated the problematic models whose realized control cost is larger than $10^5$. 
The inaccurate prediction of \method{EDMD-K}, \method{DNN-K} and \method{DD-A} (see \Cref{fig:comp_Meth}) leads to the tracking error. Finally, we note that both \method{DNN-K} and \method{EDMD-K} require an order of magnitude more data while failing to achieve the same performance as our method~\method{DD-K}, and that the performance of \method{DNN-K} is highly dependent on the pre-collected data set.

\section{Conclusions}
\label{sec:conclu}
We have developed an extended Willems' fundamental lemma for nonlinear systems that admit a Koopman linear embedding. Our results eliminate the non-trivial process of selecting lifting function and illustrate that the required size of the trajectory library is related to the dimension of the Koopman linear embedding. Future directions include developing data-driven models for nonlinear systems with approximated Koopman linear embeddings, analyzing the effect of adaptively updating the trajectory library as in~\cite{Berberich2022} and considering coupling nonlinearities between state and input  \cite{alsalti2023dataRob,haseli2023modeling}.

\begin{figure}[t]
\setlength{\abovecaptionskip}{2pt}
\subfigure[Tracking performance]{\includegraphics[width=0.22\textwidth]{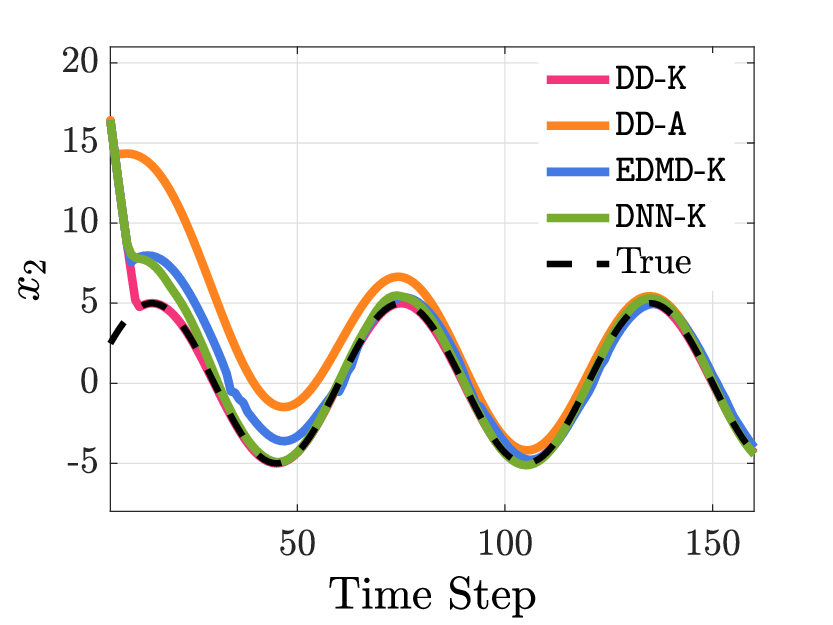}\label{fig:comp_control}}
\subfigure[Realized control cost]
{\includegraphics[width=0.22\textwidth]{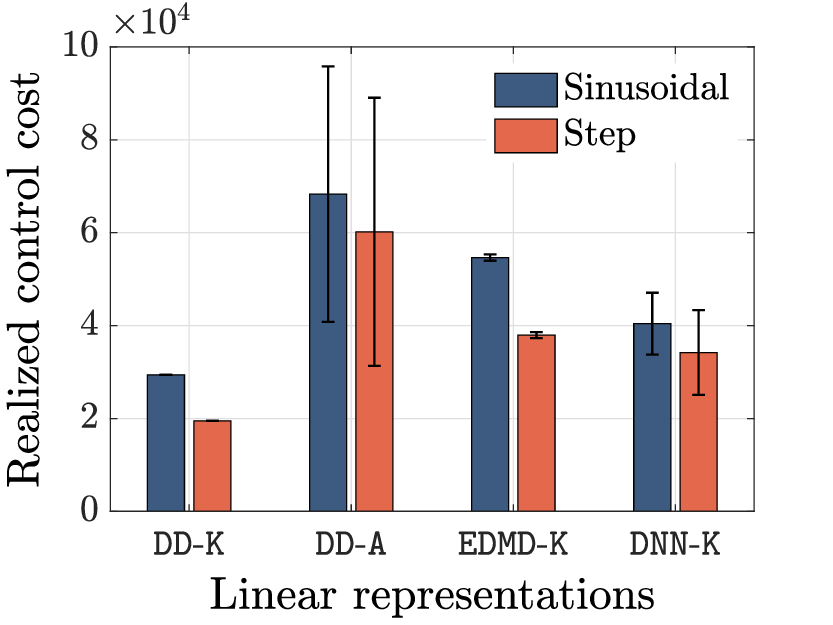}\label{fig:comp_cost}}
\caption{Control performance using different linear representations. (a) Tracking Performance. (b) Realized control costs. }
\label{fig:control}
\vspace*{-2ex}
\end{figure}

\appendix 

In this appendix, we provide proofs for \Cref{prop:traj-linear}, \Cref{prop:exist-traj-lib}, \Cref{theorem:Koopman-to-data-representation} as well as discuss the effect of data collection for applying \Cref{theorem:Koopman-to-data-representation}. We then prove the existence of a rich-enough trajectory library for \Cref{prop:exist-traj-lib}, introduce the approach and challenges in determining the existence and dimension of Koopman linear embedding. We finally propose a sufficient condition for the nonexistence of the Koopman linear embedding for a specific order.

\subsection{Technical proofs}
\subsubsection{Proof of \Cref{prop:traj-linear}}
\label{append:lemma2}
    As we do not assume the observability of the LTI system, we first derive the general form for all potential initial states $x_0$ for the given leading $L_1$ data sequence $\col(u_{0:L_1-1}, y_{0:L_1-1})$. We then present output trajectories which starts from different potential initial states $x_0$ with fixed input $u$ are the same.
    
    For the given leading $L_1$ data sequence, the output trajectory can be decomposed as follows:
   \[
   y_{0:L_1-1}=\mathcal{T}_{L_1}  u_{0:L_1-1} + \mathcal{O}_{L_1}x_0,
    \]
    where we have
    \[
    \begin{aligned}
\mathcal{T}_{L_1} &= 
\begin{bmatrix}
D & 0  & \cdots & 0 \\
C B & D & \cdots & 0 \\
CAB & CB & \cdots & 0\\
\vdots & \vdots  & \ddots & \vdots \\
CA^{L_1\!-\!2}B & CA^{L_1\!-\!3}B & \cdots & D \\ 
\end{bmatrix}\!,  \\
\mathcal{O}_{L_1}& = 
\begin{bmatrix}
    C\\
    CA\\
    \vdots\\
    CA^{L_1\!-\!1}
\end{bmatrix}.
    \end{aligned}
\]
Let $v = y_{0:L_1-1}-\mathcal{T}_{L_1} u_{0:L_1-1}$. Without loss of generality, the initial state can be represented as
\[
x_0 = x_p + \hat{x}
\]
where $x_p = \mathcal{O}_{L_1}^\dag v$ and $\hat{x} \in \mathrm{null}(\mathcal{O}_{L_1})$.

For the fixed input $u$,  the trajectory starts from the potential initial state can be derived as 
\[
\begin{aligned}
 y_{0:L-1}&=\mathcal{T}_{L}  u_{0:L-1} + \mathcal{O}_{L}x_0 \\
 &= \mathcal{T}_{L}  u_{0:L-1} + \mathcal{O}_{L} x_p + \mathcal{O}_{L} \hat{x}.   
\end{aligned}
\]
Since $\hat{x} \in \mathrm{null}(\mathcal{O}_{L_1})$, from the Cayley–Hamilton theorem, we have $\hat{x} \in \mathrm{null}(\mathcal{O}_{L})$. That leads to 
\[
y_{0:L-1} = \mathcal{T}_{L}  u_{0:L-1} + \mathcal{O}_{L} x_p,
\]
which completes the proof.

\subsubsection{Proof of \Cref{prop:exist-traj-lib}}
\label{append:them2}
    We here prove that, under the assumption of persistent excitation in \Cref{def:persistent-excitation-nonlinear}, the trajectory library \cref{eq:trajectory-library-Hankel} can fully capture the behavior of the Koopman linear embedding~\cref{eqn:Koop-rep}.
    
    Since all input-output trajectories of \eqref{eqn:nonlinear-1} are trajectories of \eqref{eqn:Koop-rep}, it is standard to derive 
    \begin{equation}
    \begin{bmatrix}
        u_\D^i  \\
        y_\D^i
    \end{bmatrix} = \begin{bmatrix}
        I_{mL} & \mathbb{0}_{mL \times n_\z} \\
        \mathcal{T}_L & \mathcal{O}_L
    \end{bmatrix} 
    \begin{bmatrix}
        u_\D^i\\
        z_0^i
    \end{bmatrix}, \qquad i = 1, \ldots, l
    \end{equation}
    where $z_0^i = \Phi(x_0^i)$ and 
    \[
    \begin{aligned}
    \mathcal{T}_L &= 
    \begin{bmatrix}
    D & 0 & \cdots & 0 \\
    CB & D  & \cdots & 0 \\
    CAB & CB  & \cdots & 0\\
    \vdots  & \vdots & \ddots & \vdots \\
    CA^{L-2}B & CA^{L-3}B  & \cdots & D \\ 
    \end{bmatrix}, \\
    \mathcal{O}_L &= 
    \begin{bmatrix}
        C\\
        CA\\
        \vdots\\
        CA^{L-1}
    \end{bmatrix}.
    \end{aligned}
    \]
    Thus, the trajectory library can be represented as 
\begin{equation} \label{eq:trajectory-library}
\begin{aligned}
H_\D &= \begin{bmatrix}
        I_{mL} & \mathbb{0}_{mL \times n_\z} \\
        \mathcal{T}_L & \mathcal{O}_L
    \end{bmatrix}
    H_K  \\&= \begin{bmatrix}
        I_{mL} & \mathbb{0}_{mL \times n_\z} \\
        \mathcal{T}_L & \mathcal{O}_L
    \end{bmatrix}\! 
    \begin{bmatrix}
        U_\D\\
        Z_0
    \end{bmatrix},    
\end{aligned}
\end{equation}
where we have 
$$U_\D = \begin{bmatrix}
     u_\D^1 & u_\D^2 & \ldots & u_\D^l
\end{bmatrix}$$ and 
$$Z_0 = \begin{bmatrix}
        \Phi(x_0^1)& \Phi(x_0^2) & \ldots & \Phi(x_0^{l}))
    \end{bmatrix}. $$

   First, it is clear that for any $g \in \mathbb{R}^l$, we have $H_\D g \in \mathcal{B}_2\!\!\mid_L$ since each column of $H_\D$ belongs to $\mathcal{B}_2\!\!\mid_L$ and $\mathcal{B}_2\!\!\mid_L$ is a linear subspace in $\mathbb{R}^{(m+p)L}$. 

   Second, let $\col(u,y) \in \mathcal{B}_2\!\!\mid_L$ be an arbitrary length-$L$ trajectory of  \eqref{eqn:Koop-rep}. We can write it as 
\begin{equation} \label{eq:any-trajectory}
    \begin{bmatrix}
        u \\ y
    \end{bmatrix} = 
    \begin{bmatrix}
        I_{mL} & \mathbb{0}_{mL \times n_\z} \\
        \mathcal{T}_L & \mathcal{O}_L
    \end{bmatrix} 
    \begin{bmatrix}
        u\\
        z_0
    \end{bmatrix},
\end{equation}
    with some initial state $z_0 \in \mathbb{R}^{n_\z}$. Since 
 $\col(U_\D, Z_0)$ has full row rank of $mL + n_\z$ by assumption, its column rank is also $mL + n_\z$ and thus there exists a $g\in \mathbb{R}^l$ such that 
 $$
 \col(U_\D, Z_0)g = \col(u, z_0).
 $$ 
 Substituting this into \cref{eq:trajectory-library,eq:any-trajectory}, we derive $ H_\D g = \col(u, y)$. This completes the proof. 
\subsubsection{Proof of \Cref{theorem:Koopman-to-data-representation}}
\label{append:them3}
We here demonstrate that we can construct an accurate data-driven linear representation for nonlinear systems that admit a Koopman linear embedding by combining \Cref{prop:exist-traj-lib,prop:traj-rela}.

Since the pre-collected data is persistently exciting of order $L$, \Cref{prop:exist-traj-lib} confirms that $\col(u_\ini, y_\ini, u_\f, y_\f)$ is a valid trajectory with length $L= T_\ini + N$ of the Koopman linear embedding \cref{eqn:Koop-rep} if and only if there exists a vector $g \in \mathbb{R}^{l}$ such that \Cref{eq:data-representation-Koopman} holds. In addition, \Cref{prop:traj-rela} guarantees that the length-$L$ trajectory $\col(u_\ini, y_\ini, u_\f, y_\f)$ of the Koopman linear embedding \eqref{eqn:Koop-rep}  is a valid trajectory of the nonlinear system \cref{eqn:nonlinear-1} if and only if $\col(u_\ini, y_\ini)$ is a trajectory of the nonlinear system \cref{eqn:nonlinear-1}, which readily holds. 
\begin{remark}
    We note that, when the nonlinear system exhibits rich behavior in its trajectory space, its associated Koopman linear embedding can have a high dimension which requires a large amount of data. If insufficient data is provided, the data-driven representation \cref{eq:data-representation-Koopman} can only capture a subspace of the true trajectory space and contain the approximation error for \cref{eqn:Koop-rep} which is similar to the analysis in \cite{alsalti2023dataRob} for utilizing infinite dimensional kernel functions to represent the nonlinearity of the system. This approach \cite{alsalti2023dataRob} has satisfactory control performance, illustrating that capturing only part of the trajectory space might be acceptable to apply Willems’ fundamental lemma-based data-driven representations in certain scenarios. We leave the perturbation analysis for our approach to future work.
\end{remark}

\subsection{Auxiliary technical discussions}
\subsubsection{Proof of the existence of trajectory set with lifted excitation of order $L$}
\label{append:exist-traj-lib}
Without loss of generality, we assume the Koopman linear embedding for the nonlinear system has linearly independent lifting functions. Then, we prove the trajectory set with lifted excitation of order $L$ always exists for the nonlinear systems that admits a Koopman linear embedding with linearly independent lifting functions.
\begin{lemma}
Suppose the nonlinear system \Cref{eqn:nonlinear-1} admits a Koopman linear embedding \Cref{eqn:Koop-rep} with linearly dependent lifting functions $\phi_1(\cdot), \ldots, \phi_{n_\z}(\cdot)$ which contains $r_\z$ linearly independent functions that are $\bar{\phi}_1(\cdot), \ldots, \bar{\phi}_{r_\z}$. Then, there exists another Koopman linear embedding associated with linearly independent lifting functions that are $\bar{\phi}_1(\cdot), \ldots, \bar{\phi}_{r_\z}$.
\end{lemma}
\begin{proof}
    Let $\bar{z}_k := \bar{\Phi}(x_k) := \col(\bar{\phi}_1(x_k),\ldots,\bar{\phi}_{r_\z}(x_k))$ and we have 
    \begin{gather*}
    \bar{z}_k = \bar{\Phi}(x_k) = P \Phi(x_k) = Pz_k, \\
    z_k =\Phi(x_k) = T \bar{\Phi}(x_k) = T\bar{z}_k, 
    \end{gather*}
    where $P$ is the permutation matrix and $T$ represents the linear dependence of $\Phi(x)$ and $\bar{\Phi}(x)$. Thus, we have
    \begin{gather*}
    \bar{z}_{k+1}  = Pz_{k+1} = PAz_k+PBu_k = PAT\bar{z}_k +PBu_k, \\
    y_k  =  C z_k + D u_k = C T\bar{z}_k + Du_k,
    \end{gather*}
    which completes our proof.
\end{proof}

\begin{proposition}
    Consider the nonlinear system \Cref{eqn:nonlinear-1} that admits a Koopman linear embedding \Cref{eqn:Koop-rep} with linearly independent lifting functions $\phi_1(\cdot), \ldots, \phi_{n_\z}(\cdot)$. There always exists a set of length-$L$ trajectories of the nonlinear system that provides lifted excitation of order $L$. 
\end{proposition}
\begin{proof}
    We aim to demonstrate there exist input-state-output trajectories that can make the matrix $H_K$ in \cref{eq:exitation-assumption} to be full row rank. It is equivalent to prove that there exists $x_1, \ldots, x_l \in \mathbb{R}^n$ ($n$ is the state dimension of the nonlinear system) such that 
    \[
    M := \begin{bmatrix}
        \Phi(x_1) & \Phi(x_2) & \cdots & \Phi(x_l)
    \end{bmatrix}
    \]
    has full row rank as the input of the sequence can be arbitrarily designed. We can first randomly sample $m>0$ points $x_1, \ldots, x_m \in \mathbb{R}^n$ and formulate the matrix $M = \begin{bmatrix}
        \Phi(x_1) & \Phi(x_2) & \cdots & \Phi(x_m)
    \end{bmatrix}$. We shall illustrate that there exists an $\bar{x} \in \mathbb{R}^n$ with the associated $\Phi(\bar{x})$ such that $\Phi(\bar{x})$ is not in the range space of $M$ if $M$ does not have full row rank. Thus, we can insert $\Phi(\bar{x})$ as a column of $M$ to increase the rank of $M$ until its reaches $n_\z$.

    Let $r < n_\z$ be the rank of $M$ and $\mathcal{N}$ be the left kernel space of $M$. $\mathcal{N}\setminus \{\mathbb{0}\}$ is not empty because $M$ does not have full row rank. We can choose $v \in \mathcal{N} \setminus \{\mathbb{0}\}$ and there exists $\bar{x} \in \mathbb{R}^n$ such that $v^\tr \Phi(\bar{x}) \neq 0$, otherwise it is contradictory with the condition $\phi_1, \ldots, \phi_n$ are linearly independent. Thus, $\Phi(\bar{x})$ is not in the range space of $M$ which indicates the rank of $\begin{bmatrix}
        \Phi(x_1)& \ldots & \Phi(x_m) & \Phi(\bar{x})
    \end{bmatrix}$ is increased to $r+1$. By iteratively conducting this procedure, we can formulate a matrix $M$ with full row rank.
\end{proof}

\subsubsection{Existence of Koopman linear embedding with dimension estimation}
\label{append:dis-KLE}
We note that given the trajectories of the nonlinear system, it is still an open problem to determine the existence of the Koopman linear embedding. It is necessary that 1) there exists an LTI system whose trajectory space is spanned by rich-enough trajectories of the nonlinear system and 2) there exists a series of associated lifting functions for bridging the origin nonlinear system and the LTI system. The first condition is studied in \cite{markovsky2022identifiability}. If trajectories of infinite length are given, one can determine the existence of the LTI system by checking the rank increasing of the Hankel matrix with its order $L$ while the existence of the lifting function remains unclear, and we leave it for our future work. On the other hand, the existence of the LTI system can not be determined with trajectories of finite length, as there always exists a trivial solution and the trajectory with finite length may not fully capture the behavior of the nonlinear system.

If the existence of the Koopman linear embedding and a sufficiently long trajectory are provided, its dimension can be obtained via utilizing the algorithm provided in [25, Algorithm 1]. The algorithm iteratively increases the depth of the trajectory library and searches its left kernel space to construct a kernel representation for the Koopman linear embedding. The dimension of the Koopman linear embedding can be obtained from the kernel representation and we consider it as ``hidden" dimension of the nonlinear system which determines the richness of the behavior of the nonlinear system. We refer interested readers to \cite{markovsky2022identifiability} for details.

\subsubsection{Sufficient condition for the nonexistence of $\bar{n}_\z$ dimensional Koopman linear embedding}
\label{append:nonexist-SC}
It is generally difficult to determine the order (upper bound) of the Koopman linear embedding for the nonlinear system as it is challenging to determine the existence of the Koopman linear embedding and requires rich enough data. We here provide a sufficient condition for the nonexistence of the Koopman linear embedding that is smaller than a specific~order. 
\begin{proposition}
    Given the input-output trajectory $u_\D, y_\D$ of the nonlinear system \cref{eqn:nonlinear-1} with length $n_\D$ and construct a trajectory library $H_\D = \col(\mathcal{H}_L(u_\D), \mathcal{H}_L(y_\D))$ with order $L$ where $L \le n_\D$. If $\mathrm{rank}(H_{\D}) > mL+\Bar{n}_{\z}$, then the Koopman linear embedding with an order no larger than $\bar{n}_{\z}$ does not~exist.
\end{proposition}
\begin{proof}
    We prove this proposition by showing $\mathrm{rank}(H_{\D}) \le mL+\bar{n}_\z$ is a necessary condition for the existence of the $n_\z$ dimensional Koopman linear predictor with $n_\z \le \bar{n}_\z$. If the $n_z$ dimensional linear predictor exists , it can be represented as   
    \[
        z_{k+1} = Az_k + Bu_k, \quad y_k = Cz_k + Du_k,
    \] 
    and $A \in \mathbb{R}^{n_\z \times n_\z}, B \in \mathbb{R}^{n_\z \times m}, C \in \mathbb{R}^{p\times n_\z}$ and $D \in \mathbb{R}^{p \times m}$. It is natural to consider each column of $H_\D$ as a length-$L$ trajectory of the Koopman linear embedding and derive $H_{\D}$ as 
    \[
    H_\D = \begin{bmatrix}
        I_{mL} & \mathbb{0}_{mL \times n_\z} \\
        \mathcal{T}_L & \mathcal{O}_L
    \end{bmatrix}\! 
    \begin{bmatrix}
        \mathcal{H}_L(u_\D)\\
        Z_0
    \end{bmatrix},
    \]
    where $Z_0$ represents the initial states of each trajectory (\emph{i.e.}, column vector of $H_\D$) and we have
    \[
    \mathcal{T}_L = 
    \begin{bmatrix}
    D & 0 & \cdots & 0 \\
    CB & D  & \cdots & 0 \\
    CAB & CB  & \cdots & 0\\
    \vdots  & \vdots & \ddots & \vdots \\
    CA^{L-2}B & CA^{L-3}B  & \cdots & D \\ 
    \end{bmatrix}, \quad \mathcal{O}_L = 
    \begin{bmatrix}
        C\\
        CA\\
        \vdots\\
        CA^{L-1}
    \end{bmatrix}.
    \]
    Since $\mathrm{rank}(H_\D) \le \mathrm{rank}(\col(\mathcal{H}_L(u_\D), Z_0)) \le mL+n_z \le mL+\bar{n}_z$, our proof is complete.
\end{proof}

\bibliographystyle{IEEEtran}
\bibliography{ref}

\end{document}